\documentclass[11pt, a4paper ,reqno]{amsart}
\usepackage{syntonly}
\usepackage{pstricks,subfigure,epsfig}
\usepackage{amsmath,amsfonts,amssymb,amsthm,euscript,upgreek}
\usepackage{enumerate}
\usepackage{multirow}
\usepackage{appendix}

\newcommand{\R}{\mathbb{R}}                     

\newcommand{\CP}{\mathbb{C}\mathrm{P}}

\newcommand{\CH}{\mathbb{C}\mathrm{H}}
\newcommand{\ol}{\mathrm{Hol}}
\newcommand{\hilb}{\mathcal{H}}
\newcommand{\C}{\mathbb{C}}                     
\newcommand{\de}{\partial}                      

\newcommand{\Ric}{\mathrm{Ric}}

\newcommand{\re}{\mathop{\mathrm{Re}}}

\newcommand{\K}{K\"{a}hler}

\newtheorem{theor}{Theorem}
\newtheorem{prop}[theor]{Proposition}

\newtheorem{lem}[theor]{Lemma}
\newtheorem{cor}[theor]{Corollary}

\newtheorem{ex}{Example}
\newtheorem{remar}[theor]{Remark}

\begin{document}
\title{Balanced metrics on Hartogs domains}
\author[A. Loi, M. Zedda]{Andrea Loi, Michela Zedda}
\address{Dipartimento di Matematica e Informatica, Universit\`{a} di Cagliari,
Via Ospedale 72, 09124 Cagliari, Italy}
\email{loi@unica.it; michela.zedda@gmail.com  }
\thanks{
The first author was supported  by the M.I.U.R. Project \lq\lq Geometric
Properties of Real and Complex Manifolds'';
the second author was  supported by  RAS
through a grant financed with the ``Sardinia PO FSE 2007-2013'' funds and 
provided according to the L.R. $7/2007$.}
\date{}

\begin{abstract}
An $n$-dimensional strictly pseudoconvex Hartogs domain $D_F$ can be equipped with a
natural \K\ metric $g_F$. 
In this paper we prove that if  $m_0g_F$ is   balanced for a given positive integer $m_0$ then   
$m_0>n$ and $(D_F, g_F)$ is holomorphically isometric  to an open subset of
the $n$-dimensional complex hyperbolic space.
\end{abstract}

\maketitle

\section{Introduction}
Let $M$ be a complex manifold  endowed with a K\"ahler metric $g$
and let $\omega$ be the \K\ form associated to $g$, i.e. $\omega (\cdot ,\cdot  )=g(\cdot, J\cdot)$.
Assume that  the metric $g$ can be described by a
strictly plurisubharmonic real valued function $\Phi :M\rightarrow\R$, called a {\em K\"ahler potential} for $g$, 
i.e.
$\omega=\frac{i}{2}\de\bar\de \Phi$.

A K\"ahler potential is not unique, in fact it is defined up to an addition with the real part of a holomorphic function on $M$. 
Let $\hilb_\Phi$ be the weighted Hilbert space of square integrable holomorphic functions on $(M, g)$, with weight $e^{-\Phi}$, namely
\begin{equation}\label{hilbertspace}
\hilb_\Phi=\left\{ f\in\ol(M) \ | \ \, \int_M e^{-\Phi}|f|^2\frac{\omega^n}{n!}<\infty\right\},
\end{equation}
where $\frac{\omega^n}{n!}=\det(\de\bar \de \Phi)\frac{\omega_0^n}{n!}$ is the volume form associated to $\omega$ and $\omega_0=\frac{i}{2}\sum_{\alpha=0}^{n-1} dz_\alpha\wedge d\bar z_\alpha$ is the standard K\"ahler form on $\C^n$.
If $\hilb_\Phi\neq \{0\}$ we can pick an orthonormal basis $\{f_j\}$ and define its reproducing kernel by
$$K_{\Phi}(z, z)=\sum_{j=0}^\infty |f_j(z)|^2 .$$
Consider the function 
\begin{equation}\label{epsilon}
\varepsilon_g(z)=e^{-\Phi(z)}K_{\Phi}(z, z). 
\end{equation} 
As suggested by the notation $\varepsilon _g$ depends only on the metric $g$ and not on the choice of the K\"ahler potential $\Phi$. In fact, if $\Phi'=\Phi-\re(\varphi)$, for some holomorphic function $\varphi$, is another potential for $\omega$, we have $e^{-\Phi'}=e^{-\Phi}|e^{\varphi}|^2$. Furthermore, since $\varphi$ is holomorphic and $\de\bar\de \Phi'=\de\bar\de \Phi$,  $e^\varphi$ is an isomorphism between $\hilb_\Phi$ and $\hilb_{\Phi'}$, and thus we can write $K_{\Phi'}(z, z)=|e^{\varphi}|^2K_{\Phi}(z, z) $, where $K_{\Phi}(z, z)$ (resp. $K_{\Phi'}(z, z)$) is the reproducing kernel of $\hilb_\Phi$ (resp. $\hilb_{\Phi'}$). It follows that $e^{-\Phi(z)}K_{\Phi}(z, z)=e^{-\Phi'(z)}K_{\Phi'}(z, z)$, as claimed. 

In the literature the function $\varepsilon_g$ was first introduced under the name of $\eta$-{\em function} by J. Rawnsley in \cite{rawnsley}, later renamed as $\varepsilon$-{\em function} in \cite{CGR}, and it is also appear under the name of {\em distortion function} for the study of Abelian varieties by  J. R. Kempf \cite{kempf} and S. Ji \cite{ji}, and for complex projective varieties by S. Zhang \cite{zhang}. It also plays a fundamental role in the geometric quantization of a K\"ahler manifold and in the Tian-Yau-Zelditch asymptotic expansion (see \cite{graloi} and references therein).

\vskip 0,3cm

\noindent
{\bf Definition.}
{\em Let $g$ be a  \K\  metric  on a complex manifold $M$  such that $\omega =\frac{i}{2}\partial\bar\partial\Phi$. The metric $g$ is \emph{balanced} if  the function $\varepsilon_g$ is a positive  constant.}



\begin{remar}\rm
 The definition of balanced metrics
 was originally given by S. Donaldson \cite{donaldson} in the case of  a compact polarized \K\ manifold $(M, g)$ and generalized in \cite{arezzoloi}
 (see also \cite{cucculoibal}, \cite{grecoloi}, \cite{englisweigh}). 
 In the compact case   the potential $\Phi$ will certainly not exist globally and the only holomorphic functions on $M$ are the constants.
 Nevertheless, since $g$ is polarized there exists an  hermitian  line bundle $(L, h)\rightarrow M$ such that  $\Ric (h)=\omega$.
 One can then endowed the space of global holomorphic sections of $L$, denoted by  $H^0(L)$, 
 with the scalar product
 $$\langle  s,t \rangle_h= \int_M h(s(x),t(x))\frac{\omega^n}{n!} , s, t\in H^0(L).$$
 If $H^0(L)\neq\{0\}$ one can set  
$$\varepsilon_g(x)=\sum_{j=0}^Nh(s_j(x),s_j(x)),$$
where $\{s_0, \dots, s_N\}$, $N+1=\dim H^0(L)$, is an orthonormal basis of  $(H^0(L), \langle , \rangle_h)$
and  define the metric $g$ {\em balanced} if $\varepsilon_g$ is a positive constant.
\end{remar}

In this paper we study  the  balanced condition  for a particular class of  strictly pseudoconvex  domains $D_F$ of $\C^n$, called {\em Hartogs domains} (see next section or  
\cite{englis}), equipped with a \K\ metric $g_F$ depending on a real valued function $F$. Our main result
is Theorem \ref{hartogs} below where we prove that if  the metric $m_0g_F$ of a $n$-dimensional Hartogs domain $D_F$ is balanced for a given
$m_0>n$, then $(D_F, g_F)$ is holomorphically isometric
to  an open subset  of  the $n$-dimensional  complex hyperbolic space.
The paper contains another section with the description of the  Hartogs domains
and the proof of Theorem \ref{hartogs}.

\section{Statement and proof of the main result}

Let $x_0 \in \R^+ \cup \{ + \infty \}$ and let $F: [0, x_0)
\rightarrow (0, + \infty)$ be a decreasing continuous function,
smooth on $(0, x_0)$. The Hartogs domain $D_F\subset
{\C}^{n}$ associated to the function $F$ is defined by
$$D_F = \{ (z_0, z_1,\dots ,z_{n-1}) \in {\C}^{n} \; | \; |z_0|^2 < x_0, \ ||z||^2  < F(|z_0|^2)
\}, $$
where $||z||^2=|z_1|^2+\dots + |z_{n-1}|^2$.
We shall assume that the natural $(1, 1)$-form on $D_F$  given by
\begin{equation}\label{omegaf}
\omega_F = \frac{i}{2} \partial \overline{\partial}
\log \left(\frac{1}{F(|z_0|^2) - ||z||^2 }\right),
\end{equation}
 is  a \K\ form on $D_F$. The following proposition gives some
conditions on $D_F$ equivalent to this assumption:
\begin{prop}[\cite{canonical}]\label{Kmet}
Let $D_F$ be a Hartogs domain in ${\C}^{n}$. Then the
following conditions are equivalent:
\begin{itemize}
\item [(i)] the $(1, 1)$-form $\omega_F$  given by (\ref{omegaf})
is a \K\ form; \item [(ii)] the function $- \frac{x F'(x)}{F(x)}$
is strictly increasing, namely $-\left( \frac{x F'(x)}{F(x)}\right)' >0$
for every $x \in [0, x_0)$; \item [(iii)] the boundary of $D_F$ is
strongly pseudoconvex at all $z = (z_0, z_1,\dots,z_{n-1})$ with
$|z_0|^2 < x_0$.
\end{itemize}
\end{prop}
The \K\ metric $g_F$ associated to the \K\ form $\omega_F$ is the metric we will be dealing with in the present paper. It follows by (\ref{omegaf}) that a K\"ahler potential for this metric is given by
$$\Phi_F=-\log \left(F(|z_0|^2) - ||z||^2\right).$$

\begin{ex}\label{uno}\rm
When  $F(x)=1-x, 0\leq x< 1$, 
$$D_F=\CH^n=\{(z_0, z_1,\dots ,z_{n-1})\ |\ |z_0|^2+\|z\|^2<1)\},$$
  the $n$-dimensional complex hyperbolic
space $\CH^n$ and $g_F$ is the hyperbolic metric, i.e.
$g_F=g_{hyp}$. 
A K\"ahler potential for $g_{hyp}$ is given by $\Phi_{hyp}=-\log(1-\sum_{\alpha=0}^{n-1}|z_\alpha|^2)$, and the associated  volume form reads
$$\frac{\omega_{hyp}^n}{n!}=\left(1-\sum_{\alpha=0}^{n-1}|z_\alpha|^2\right)^{-(n+1)}\frac{\omega_0^n}{n!}.$$ 
Consider  $m \,g_{hyp}$, for  a positive integer $m$, 
and let  $\hilb_{m\Phi_{hyp}}$ be the weighted Hilbert space of square integrable holomorphic functions on $(\CH^n,  m \,g_{hyp})$, with weight  $e^{-m\Phi_{hyp}}=\left(1-\sum_{\alpha=0}^{n-1}|z_\alpha|^2\right)^{ m }$, namely
$$\hilb_{m\Phi_{hyp}}=\left\{ \varphi\in\ol(\CH^n)\ | \ \int_{\CH^n}\left(1-\sum_{\alpha=0}^{n-1}|z_\alpha|^2\right)^{ m -(n+1)}|\varphi|^2\frac{\omega_0^n}{n!}<\infty \right\}.$$
If $m\leq n$, then it is not hard to see that  $\hilb_{m\Phi_{hyp}}=\{0\}$. On the other hand, for $m>n$,  an orthonormal basis  for $\hilb_{m\Phi_{hyp}}$ is given by 
$$\left\{\dots, \frac{\sqrt{( m +j-1)!}}{\sqrt{\pi^n}\sqrt{j_1!\cdots j_{n-1}!( m -n-1)!}}z_0^{j_0}\cdots z_{n-1}^{j_{n-1}},\dots\right\}.$$
where $j=j_0+\dots+j_{n-1}$.
In fact, since the metric depends only on the squared module of the variables, it is easy to see that the monomials $z_0^{j_0}\cdots z_{n-1}^{j_{n-1}}$ are a complete orthogonal system for $\hilb_{m\Phi_{hyp}}$. Further, the following computation
\begin{equation}
\begin{split}
&\int_{\CH^n} |z_0^{j_0}\cdots z_{n-1}^{j_{n-1}}|^{2}\left(1-\sum_{\alpha=0}^{n-1}|z_\alpha|^2\right)^{ m -(n+1)}\frac{i^n}{2^n}dz_0\wedge d\bar z_0\wedge\dots\wedge dz_{n-1}\wedge d\bar z_{n-1} \\
=&\pi^n\int_0^1\cdots \int_0^{1-r_1-\dots-r_{n-1}}r_0^{j_0} \cdots r_{n-1}^{j_{n-1}}\left(1-\sum_{\alpha=0}^{n-1} r_\alpha^{j_\alpha}\right)^{ m -(n+1)}dr_0 \cdots dr_{n-1}\\
=&\pi^n\frac{j_0!\cdots j_{n-1}!( m -n-1)!}{( m +j-1)!},
\end{split}\nonumber
\end{equation}
justifies the choice of the normalization constants. 
The reproducing kernel for $\hilb_{m\Phi_{hyp}}$ is then given by
$$K_{m\Phi_{hyp}}(z, z)=\frac{( m -1)\cdots( m -n)}{\pi^n (1-\sum_{j=0}^{n-1}|z_j|^2)^ m },$$
and thus 
$$\varepsilon_{ m g_{hyp}}(z)=\frac{( m -1)\cdots( m -n)}{\pi^n}.$$
\end{ex}

In this  example we have  that  the metric $m \,g_{hyp}$ is balanced iff   $m>n$. 
In the geometric quantization framework introduced in \cite{CGR} 
the \K\ forms satisfying this property play a fundamental role for the quantization by deformation of the \K\ manifold
$(M, g)$. In our setting one says that a  \K\ manifold $(M, g)$ admits a {\em regular quantization} if  the functions 
$$\varepsilon_{mg} (z)=e^{-m\Phi (z)}K_{m\Phi}(z, z)$$
are positive constants (depending on $m$) for all   sufficiently large positive  integers.

Regarding regular quantizations we  have  the following lemma which will be an important ingredient in the proof of our main result, Theorem \ref{hartogs}.

\begin{lem}\label{constsc}
Let $g$ be a K\"ahler metric on a complex manifold $M$. If $(M, g)$  admits a  regular quantization  then the scalar curvature
of the metric $g$ is constant.
\end{lem}
\begin{proof}
See Theorem 5.3 in \cite{arezzoloiquant} for the compact case and  Theorem 4.1 in  \cite{quant} for the noncompact one.
\end{proof}


Hartogs domains have been
considered in \cite{englis} and \cite{quant} in the framework of
quantization of \K\ manifolds. In \cite{cucculoiglob} is studied the
existence of global symplectic coordinates on $(D_F, \omega_F)$
and \cite{riemhartogs} deals with the Riemannian geometry of
$(D_F, g_F)$.
In \cite{quant} (see also  \cite{canonical}) these domains are studied from the scalar curvature viewpoint.
The main results obtained in  \cite{canonical}  and in  \cite{riemhartogs} are summarized in the following two lemmata  needed in the proof of Theorem \ref{hartogs}
and its Corollary \ref{corolhartogs}.

\begin{lem}\label{scalcurv}
Let $(D_F , g_F )$ be a n-dimensional Hartogs domain. Assume 
that its scalar curvatures is constant. Then $(D_F , g_F )$ is 
holomorphically isometric to  an open subset  of  the complex hyperbolic space $(\CH^n, g_{hyp})$.
\end{lem}

\begin{lem}\label{lemcomplete}
A Hartogs domain $(D_F, g_F)$  is geodesically complete
if and only if
\begin{equation}\label{complcond}
\int_0^{\sqrt{x_0}} \sqrt{- \left(\frac{xF'}{F}\right )'}|_{x=
u^2} \ du = +\infty,
\end{equation}
where we define  $\sqrt{x_0}=+\infty$ for $x_0=+\infty$.
\end{lem}
For the proof of Theorem \ref{hartogs} we need another result,   Lemma \ref{englislemma} below, 
which is a straightforward  generalization to dimension $n$ of Propositions 3.12 and 3.14 proven by M. Engli$\rm\check{s}$ in \cite{englis}.
In order to state it  we set
\begin{equation}\label{ck}
c_k(F^ m )=\int_0^{x_0}t^kF(t)^mG(t)dt,
\end{equation}
where
\begin{equation}\label{G(t)}
G(t)=-\left(\frac{tF'}{F}\right)',
\end{equation}
(notice that $G(t)>0$ by $(ii)$ in Proposition \ref{Kmet})
and assume that there exists a real number $\gamma$ such that for all positive integers $ m $
\begin{equation}\label{ckcompute}
\sum_{k=0}^\infty\frac{t^k}{c_k(F^ m )}=( m -1+\gamma)F(t)^{- m }.
\end{equation}
Many examples of Hartogs domains satisfy this condition  (see \cite[pp. 450-454]{englis}). Such domains   admit a  quantization by deformation  (see \cite{englis} for details) and so they are also  interesting from the physical point of view. 

Let us  also write the volume element corresponding to the metric $\omega_F$ by
\begin{equation}\label{volumelement}
\frac{\omega_F^n}{n!}=\frac{F^2(|z_0|^2)}{(F(|z_0|^2) -||z||^2)^{n+1}}\,G(|z_0|^2)\frac{\omega_0^n}{n!}.
\end{equation}

\begin{lem}\label{englislemma}
Let $(D_F,g_F)$ be an Hartogs domain and let $\hilb_ {m\Phi_F}$ 
be the corresponding  weighted Hilbert space given by (\ref{hilbertspace}).
Assume that condition (\ref{ckcompute}) is satisfied for all positive integers $m$. 
Then $\hilb_ {m\Phi_F}\neq \{0\}$ iff $m >n$ and its reproducing kernel  is given by
\begin{equation}
K_{ m\Phi_F} (z, z)=\frac{( m -2)\cdots( m -n)}{\pi^n (F(|z_0|^2)-||z||^2)^ m }\left[ m -1+(1-w) \gamma\right],\nonumber
\end{equation}
where $w=\frac{\|z\|^2}{F(|z_0|^2)}$ and $\gamma$ is the real number appearing in 
(\ref{ckcompute}).
\end{lem}
\begin{proof}
It is not hard to verify that the monomials $z_0^{j_0}z_1^{j_1}\cdots z_{n-1}^{j_{n-1}}$ are a complete orthogonal system for $\hilb_{m\Phi_F}$, for $m>n$. Hence, the well-known formula for reproducing kernels gives for the Hilbert space $\hilb_{m\Phi_F} $
\begin{equation}\label{eqKm}
K_{ m\Phi_F} (z, z)=\sum_{j_0,\dots,j_{n-1}}\frac{|z_0|^{2j_0}\cdots |z_{n-1}|^{2j_{n-1}}}{\| z_0^{j_0} \cdots z_{n-1}^{j_{n-1}} \|_ m ^2},
\end{equation}
where
$$\| z_0^{j_0} \cdots z_{n-1}^{j_{n-1}} \|_ m ^2=\int_{D_F}\left(F(|z_0|^2) -||z||^2\right)^{ m }\prod_{k=0}^{n-1}|z_k|^{2j_k}\frac{\omega_F^n}{n!}.$$
By formula (\ref{volumelement}) the right hand side is equal to
\begin{equation}
\int_{D_F}\left(F(|z_0|^2) - ||z||^2 \right)^{ m -n-1}\prod_{k=0}^{n-1}|z_k|^{2j_k}F^2(|z_0|^2)G(|z_0|^2)\frac{\omega_0^n}{n!}, \nonumber
\end{equation}
which passing to polar coordinates reads
\begin{equation}
\pi^n\int_0^{{x_0}^{1/2}}\!\!\!\! \int_{0}^{F(r_0^2)^{1/2}}\!\!\!\!\!\!\!\!\cdots\int_0^{(F(r_0^2)-\sum_{i=2}^{n-1}r^2_i)^{1/2}}\!\!\!\!\!\!\!\left(F(r_0^2) - r^2\right)^{ m -n-1}\prod_{k=0}^{n-1}r_k^{2j_k}F^2(r_0^2)G(r_0^2)2^ndrdr_0, \nonumber
\end{equation}
where $r^2=r_1^2+\dots+r_{n-1}^2$, $dr=dr_1\cdots dr_{n-1} $. Making now the substitution $r_i^2=t_i$ and using again the short notation $t=t_1+\dots +t_{n-1}$,  $dt=dt_1\cdots dt_{n-1}$, we get\begin{equation}
\pi^n\!\!\int_0^{{x_0}} \int_{0}^{F(t_0)}\!\!\!\!\cdots\int_0^{F(t_0)-\sum_{i=2}^{n-1}t_i}\!\!\!\left(F(t_0) - t\right)^{ m -n-1}\prod_{k=0}^{n-1}t_k^{j_k}F^2(t_0)G(t_0)dtdt_0, \nonumber
\end{equation}
which substituting $t_k=w_kF(t_0)$ for $k=1,\dots,n-1$, becomes
\begin{equation}
\pi^n\!\!\int_0^{{x_0}} \!\!\!   t_0^{j_0} F(t_0)^{ m +j} G(t_0)  dt_0\! \int_{0}^{1}\!\!\!\cdots\!\!\int_0^{1-\sum_{i=2}^{n-1}w_i}\!\!\!\!\!\left(1 - w\right)^{ m -n-1}\prod_{k=1}^{n-1}w_k^{j_k}dw, \nonumber
\end{equation}
where again $w=w_1+\dots +w_{n-1}$, $dw=dw_1\cdots dw_{n-1}$. If $m\leq n$ the last integral diverges, so we can assume $m>n$. Therefore,
\begin{equation}\label{normahartogs}
\| z_0^{j_0} \cdots z_{n-1}^{j_{n-1}} \|_ m ^2=\pi^n\frac{j_1!\cdots j_{n-1}!( m -n-1)!}{( m +j-2)!}c_{j_0}(F^{ m +j}),
\end{equation}
where  $j=j_1+\dots+j_{n-1}$ and $c_{j_0}(F^{ m +j})$ is defined by (\ref{ck}).
Thus
\begin{equation}
K_{m\Phi_F} (z, z)=\!\!\!\sum_{j_0,\dots,j_{n-1}}\!\!\!|z_0|^{2j_0}\cdots |z_{n-1}|^{2j_{n-1}}\frac{( m +j-2)!}{\pi^n j_1!\cdots j_{n-1}!( m -n-1)!}\left(c_{j_0}(F^{ m +j})\right)^{-1}.\nonumber
\end{equation}
By (\ref{ckcompute}) we can carry out the summation over $j_0$, getting
{\small\begin{equation}
\begin{split}
K_{m\Phi_F}(z, z)=&\sum_{j_1\dots,j_{n-1}}\!\!\!|z_1|^{2j_1}\cdots |z_{n-1}|^{2j_{n-1}}\frac{( m +j-2)!( m +j-1+\gamma)}{\pi^n j_1!\cdots j_{n-1}!( m -n-1)!}F^{- m -j}(|z_0|^2)\\
=&\sum_{j_1\dots,j_{n-1}}\!\!\!\frac{|z_1|^{2j_1}}{F^{j_1}(|z_0|^2)}\cdots \frac{|z_{n-1}|^{2j_{n-1}}}{F^{j_{n-1}}(|z_0|^2)}\frac{( m +j-2)!( m +j-1+\gamma)}{\pi^n j_1!\cdots j_{n-1}!( m -n-1)!}F^{- m }(|z_0|^2)\\
=&\sum_{j_1\dots,j_{n-1}}\!\!\!w_1^{j_1}\cdots w_{n-1}^{j_{n-1}}\frac{( m +j-2)!( m +j-1+\gamma)}{\pi^n j_1!\cdots j_{n-1}!( m -n-1)!}F^{- m }(|z_0|^2)\\
=&\frac{( m -2)\cdots( m -n)}{\pi^n}\!\!\!\sum_{j_1\dots,j_{n-1}}\!\!\!\frac{w_1^{j_1}}{j_1!}\cdots \frac{w_{n-1}^{j_{n-1}}}{j_{n-1}!}\left[{ m +j-1\choose  m -1}( m -1)+\right.\\
&\qquad \qquad \qquad \qquad\qquad \qquad\qquad \qquad \qquad \left.+{ m +j-2\choose  m -2} \gamma\right]F^{- m }(|z_0|^2)\\
=&\frac{( m -2)\cdots( m -n)}{\pi^n}\left[\frac{ m -1}{(1-w)^ m }+\frac{\gamma}{(1-w)^{ m -1}}\right]F^{- m }(|z_0|^2)\\
=&\frac{( m -2)\cdots( m -n)}{\pi^n}\left[\frac{ m -1}{(F(|z_0|^2)-||z||^2)^ m }+\frac{(1-w) \gamma}{(F(|z_0|^2)-||z||^2)^ m }\right]\\
=&\frac{( m -2)\cdots( m -n)}{\pi^n (F(|z_0|^2)-||z||^2)^ m }\left[ m -1+(1-w) \gamma\right].
\end{split}\nonumber
\end{equation}}
\end{proof}

We can now   state and prove  our main result, which characterizes the hyperbolic space among Hartogs domains  in terms of a balanced condition.

\begin{theor}\label{hartogs}
Let  $(D_F, g_F)$ be  a $n$-dimensional   Hartogs domain. Assume that condition (\ref{ckcompute}) is satisfied for all positive integers $m$. 
If $m _0g_F$ is balanced then $m_0>n$ and 
$(D_F, g_F)$ is holomorphically isometric to  an open subset  of  the complex hyperbolic space $(\CH^n, g_{hyp})$.
\end{theor}
\begin{proof}
Since by Lemma \ref{englislemma} $\hilb_{m\Phi_F}=\{0\}$ for $m_0\leq n$, we can set $m_0>n$. Assume that $m_0g_F$ is balanced, 
namely $e^{m_0\Phi_F}=c_{m_0}K_{m_0\Phi_F}$, 
for some positive constant $c_{m_0}$. Therefore, 
\begin{equation}\label{epsilonF}
\left(F(|z_0|^2) - ||z||^2\right)^{ -m_0 }=c_{m_0}K_{m_0\Phi_F}(z, z).\nonumber
\end{equation}
By Lemma \ref{englislemma} we get
\begin{equation}
\left(F(|z_0|^2) - ||z||^2\right)^{ -m_0 }=c_{m_0}\frac{( m_0 -2)\cdots( m_0 -n)}{\pi^n (F(|z_0|^2)-||z||^2)^ {m_0} }\left[ m_0 -1+(1-w) \gamma\right],\nonumber
\end{equation}
that is
\begin{equation}
\pi^n=c_{m_0}( m_0 -2)\cdots( m_0 -n)\left[ m_0 -1+(1-w) \gamma\right],\nonumber
\end{equation}
which yelds $\gamma=0$, being $(1-w) \gamma$ the only term depending on the variables.
Since $\gamma$ is fixed for all $m$, it follows that the reproducing kernel
of   $\hilb_ {m\Phi_F}$, for $m>n$, is given by
$$K_{{m}\Phi_F}(z, z)=\frac{({m}-1)({m}-2)\cdots({m}-n)}{\pi^n (F(|z_0|^2)-||z||^2)^{{m}}}.$$
By (\ref{epsilon}), we have 
$$\varepsilon_{{m}g_F}(z)=K_{{m}\Phi_F}(z, z)\left(F(|z_0|^2) - ||z||^2 \right)^{{m}}=\frac{({m}-1)({m}-2)\cdots({m}-n)}{\pi^n}.$$ 
Hence, for all $m>n$,  $(D_F, g_F)$ admits a regular quantization. By Lemma \ref{constsc} and Lemma 
\ref{scalcurv} above, $(D_F, g_F)$ is then holomorphically isometric to an open subset of the complex hyperbolic space. 
\end{proof}

Combining Lemma \ref{lemcomplete} with Theorem \ref{hartogs} one gets:

\begin{cor}\label{corolhartogs}
Let  $(D_F,g_F)$ be an $n$-dimensional   Hartogs domain. Assume that conditions   (\ref{complcond}) and   (\ref{ckcompute}) are satisfied (the latter  for all positive integers $m$).
If for some  positive integer $m_0$,  $m _0g_F$ is  a balanced metric then 
$(D_F, g_F)$ is holomorphically isometric to  the complex hyperbolic space $(\CH^n,g_{hyp})$.
\end{cor}

\begin{remar}\rm
A balanced metric $g$ on a complex manifold $M$ is projectively induced.  Indeed, there exists a 
holomorphic map $f:M\rightarrow \CP^{\infty}$, called  the {\em coherent states map} in J. Rawnsley terminology \cite{rawnsley},
into the  infinite dimensional complex projective space $\CP^{\infty}$
such that $f^*g_{FS}=g$, where $g_{FS}$ denotes the Fubini--Study metric on $\CP^{\infty}$ 
 (see \cite{arezzoloi} for details).
Not all projectively induced metrics are balanced. Indeed, there exist $n$-dimensional  Hartogs domains $(D_F,g_F)$, $D_F\neq \CH^n$, where $m_0g_F$ is projectively induced for $m_0>n$. An example is given by the so called {\em Springer domain} $(D_F, g_F)$ corresponding to the function  $F(x)=e^{-x}, x\in [0, +\infty)$  
(see \cite{hartogs})). Moreover, it is not hard to verify that  this domain satisfies  condition (\ref{ckcompute}) in Theorem \ref{hartogs} with $\gamma =1$  (see also  \cite{englis}). 
This shows that the condition that $m_0 g_F$ is balanced in Theorem \ref{hartogs} cannot be replaced by the weaker condition that 
$m_0 g_F$ is projectively induced.
\end{remar}

\end{document}